\documentclass[11pt,reqno]{amsart}
\usepackage[T1]{fontenc}
\usepackage{lmodern}
\usepackage{amsmath,amssymb,amsthm}
\usepackage[a4paper,margin=1.08in]{geometry}
\usepackage{microtype}
\usepackage[hidelinks]{hyperref}

\hypersetup{
  pdftitle={On the Erd\H{o}s--Ginzburg--Ziv Problem for $\mathbb{Z}_n^r$},
  pdfauthor={Dongchun Han and Hanbin Zhang},
  pdfsubject={Zero-sum sequences, polynomial moments, and truncated polynomial rings},
  pdfkeywords={Erdos-Ginzburg-Ziv constant, zero-sum sequence, polynomial method}
}

\newtheorem{theorem}{Theorem}[section]
\newtheorem{lemma}[theorem]{Lemma}
\newtheorem{proposition}[theorem]{Proposition}
\theoremstyle{remark}

\newtheorem{prob}{Problem}
\numberwithin{equation}{section}

\newcommand{\N}{\mathbb{N}}

\newcommand{\Fp}{\mathbb{F}_p}

\newcommand{\cP}{\mathcal{P}}
\newcommand{\EGZ}{\mathsf{s}}

\DeclareMathOperator{\aff}{aff}

\title[On the Erd\H{o}s--Ginzburg--Ziv Problem for $\mathbb{Z}_n^r$]{On the Erd\H{o}s--Ginzburg--Ziv Problem for $\mathbb{Z}_n^r$}

\author{Dongchun Han}
\address{School of Mathematics, Southwest Jiaotong University, Chengdu 610000, P.R. China}
\email{handongchun@swjtu.edu.cn}

\author{Hanbin Zhang}
\address{School of Mathematics (Zhuhai), Sun Yat-sen University, Zhuhai 519082, Guangdong, P.R. China}
\email{zhanghb68@mail.sysu.edu.cn}

\date{}

\begin{document}

\begin{abstract}
The Erd\H{o}s--Ginzburg--Ziv problem is a classical extremal
problem in discrete geometry. We prove that, for all positive
integers $n$ and $r$, every sequence of $4^r n$ lattice points
in $\mathbb{Z}^r$ contains $n$ points whose centroid is again
a lattice point.
\end{abstract}

\maketitle

\section{Introduction}\label{sec:introduction}

Let $\mathbb{Z}_n^r$ be the direct sum of $r$ copies of
$\mathbb{Z}_n$. The Erd\H{o}s--Ginzburg--Ziv problem asks
for the smallest positive integer $t$ (denoted by
$\EGZ(\mathbb Z_n^r)$) such that among any $t$ lattice
points in $\mathbb{Z}^r$ one can always find $n$ points
whose centroid is again a lattice point. Determining
$\EGZ(\mathbb Z_n^r)$ is a classical extremal problem
in discrete geometry, combinatorics, and number theory.

The classical theorem of Erd\H{o}s, Ginzburg and
Ziv~\cite{EGZ} asserts that
$\EGZ(\mathbb{Z}_n)=2n-1$.
For rank two, Reiher~\cite{Reiher} proved that
$\EGZ(\mathbb{Z}_n^2)=4n-3$, confirming Kemnitz'
conjecture. For general background, we refer to the
survey of Gao and Geroldinger~\cite{GG}.
For groups $\mathbb{Z}_n^r$ of arbitrary rank,
a natural question is how the coefficient of $n$
in an upper bound for $\EGZ(\mathbb{Z}_n^r)$
depends on $r$.
For $n\geq2$, the construction of Harborth~\cite{Harborth},
taking $n-1$ copies of each vector in $\{0,1\}^r$,
gives
\(
  \EGZ(\mathbb{Z}_n^r)\geq 2^r(n-1)+1
\). This lower bound is not always sharp. For example,
Elsholtz~\cite{Elsholtz} proved that
$\EGZ(\mathbb{Z}_n^3)\geq9n-8$
for every odd $n\geq3$.

Alon and Dubiner~\cite{AD} proved that, for fixed
$r\geq2$,
$\EGZ(\mathbb{Z}_n^r)\leq (cr\log r)^r n$,
where $c$ is an absolute constant.
They also posed the following problem
\cite[Section~3]{AD}.
\begin{prob}\label{Problem1}
Is there an absolute constant $c$ such that
$\EGZ(\mathbb{Z}_n^r)\leq c^r n$?
\end{prob}

Recently, Zakharov~\cite[Theorem~1.1]{Zakharov}
proved that, for fixed $r$,
$\EGZ(\mathbb{Z}_p^r)\leq 4^r p$
for all sufficiently large primes $p$.
He also proved that
$\EGZ(\mathbb{Z}_n^r)\leq 4^r n$
when all prime divisors of $n$ are sufficiently
large in terms of $r$.

On the other hand, there is also substantial progress
when the modulus $n$ is fixed and the rank $r$
tends to infinity.
Naslund~\cite{Naslund} used the partition-rank method
to obtain exponential bounds for
$\EGZ(\mathbb{Z}_p^r)$.
He also proved~\cite[Theorem~3]{Naslund} that $\EGZ(\mathbb{Z}_q^r)\leq(q-1)4^r+1$
for a prime power $q$, under the assumption that
$\mathbb{Z}_q^r$ has Property~D
(see \cite[Section~7]{GG}). Later, Fox and Sauermann \cite{FS} relate $\EGZ(\mathbb{Z}_p^r)$ to the cap set problem and proved that $\EGZ(\mathbb{Z}_p^r)\le 2p~\mathsf r(\mathbb Z_p^r)$ (where $\mathsf r(\mathbb Z_p^r)$ is the largest size of a cap-set in $\mathbb Z_p^r$), which slightly improved Naslund's result. Recently, Sauermann~\cite[Corollary~1.2]{Sauermann}
proved, for every fixed prime $p\geq5$, a bound
of the form
$\EGZ(\mathbb{Z}_p^r)\leq A_p(2\sqrt p)^r$,
where $A_p$ depends only on $p$.
More recently, Sauermann and
Zakharov~\cite[Theorem~1.1]{SZ}
proved that, for every $\varepsilon>0$
and every fixed $n\geq2$, $\EGZ(\mathbb{Z}_n^r)\leq D_{\varepsilon,n}(C_\varepsilon n^\varepsilon)^r,$
where $C_\varepsilon$ depends only on $\varepsilon$
and $D_{\varepsilon,n}$ depends only on
$\varepsilon$ and $n$.
These results give strong bounds in the regime
of fixed modulus and large rank.
Sauermann and Zakharov~\cite[Question~1.3]{SZ}
also asked whether there exists an absolute constant
$c$ such that, for every fixed $n\geq2$,
$\EGZ(\mathbb{Z}_n^r)\leq D_n c^r$
for all $r$, where $D_n$ depends only on $n$.

In this paper, we give a uniform bound for
$\EGZ(\mathbb{Z}_n^r)$ valid for every $n$ and
every $r$. Our main result is the following.

\begin{theorem}\label{thm:main}
Let $n,r\in\N$. Then
\begin{equation}\label{eq:main}
  \EGZ(\mathbb{Z}_n^r)
  \leq \binom{2r}{r}(n-1)+1
  \leq 4^r n.
\end{equation}
\end{theorem}

In particular, our result answers both Problem \ref{Problem1} above and Question 1.3 in \cite{SZ}.
Our proof combines the polynomial method with
elementary algebraic techniques.

Section~\ref{sec:preliminaries} contains the background
and auxiliary lemmas. In Section~\ref{sec:proof},
we first establish the prime-modulus bound and then
prove Theorem~\ref{thm:main}.

\section{Preliminaries}\label{sec:preliminaries}

Let $\N$ denote the set of positive integers, and write $[1,m]=\{1,\ldots,m\}$
for $m\in\N$. A sequence over $\mathbb Z_n^r$ is a finite multiset of
elements of $\mathbb Z_n^r$. The Erd\H{o}s--Ginzburg--Ziv constant
$\EGZ(\mathbb Z_n^r)$ is the least positive integer $t$
such that every sequence of $t$ elements of
$\mathbb Z_n^r$ contains $n$ terms whose sum is zero. For more background and related studies on the Erd\H{o}s--Ginzburg--Ziv constant, we refer to \cite{FGZ,GG,GHZ,GeroldingerHalterKoch2006}. Our approach combines the polynomial method (see, e.g., \cite{Guth,Petrov} for related discussion), classical group ring method \cite{OlsonI}, and  elementary algebraic techniques.

We recall the affine-space terminology used below. Let $K$ be
a field and let $W$ be a finite-dimensional vector space over
$K$. 
An affine space in $W$ is a translate
$H=w_0+U=\{w_0+u:u\in U\}$ of a linear subspace $U$;
we denote $\dim H=\dim_K U$. In particular, $W$ itself is an
affine space.

An affine combination of points $w_1,\ldots,w_t$ is a sum
$\sum_{i=1}^t\lambda_iw_i$ with $\lambda_i\in K$ and
$\sum_{i=1}^t\lambda_i=1$. Their affine hull is the set of
all affine combinations, denoted by
\[
 \aff\{w_1,\ldots,w_t\}
 =\left\{\sum_{i=1}^t\lambda_iw_i:
       \lambda_i\in K,\ \sum_{i=1}^t\lambda_i=1\right\}.
\]
Equivalently, it is
$w_1+\operatorname{span}_K\{w_i-w_1:2\leq i\leq t\}$,
where $\operatorname{span}_K$ denotes the set of all linear
combinations over $K$.
The points $w_0,\ldots,w_d$ are affinely independent if
$w_1-w_0,\ldots,w_d-w_0$ are linearly independent. They form
an affine basis of $H$ if they are affinely independent and
their affine hull is $H$. Thus an affine basis of a
$d$-dimensional affine space has $d+1$ points. Every finite
nonempty family contains an affine basis of its affine hull;
in dimension zero, an affine basis consists of one point.

The following observation produces disjoint blocks with a
common point in their affine hulls.

\begin{lemma}\label{lem:affine}
Let $v_1,\ldots,v_N$ be indexed points in an affine
space over a field, and suppose that their affine
hull has dimension at most $b-1$, where $N,b\in\N$.
Then $[1,N]$ has a partition into
$M\geq\lceil N/b\rceil$ nonempty sets
$B_1,\ldots,B_M$, each of size at most $b$, such that
\[
\aff\{v_i:i\in B_1\}
\supseteq\cdots\supseteq
\aff\{v_i:i\in B_M\}.
\]
\end{lemma}

\begin{proof}
Choose an affine basis of the given points, let
$B_1$ be its set of indices, and remove those indices.
Repeat with the remaining points until none remain. Each block is an affine basis of the points remaining
when it is chosen. Hence its affine hull contains
all subsequent blocks. Since each block has at most
$b$ elements, $N=\sum_{j=1}^{M}|B_j|\leq bM,$
as required.
\end{proof}

We next introduce the polynomial algebra used in the proof. Let
$p$ be a prime, let $r\in\N$, and put
\begin{equation}\label{eq:ring}
  R=\Fp[X_1,\ldots,X_r]/(X_1^p-1,\ldots,X_r^p-1),
  \qquad J=(X_1-1,\ldots,X_r-1)\subseteq R.
\end{equation}
For $a=(a_1,\ldots,a_r)\in\Fp^r$, we write $X^a=X_1^{a_1}\cdots X_r^{a_r}\in R$. Here and below, exponents are taken modulo $p$. The relations in $R$ give
\begin{equation}\label{eq:monomial-products}
  X^aX^c=X^{a+c},\qquad (X^a)^p=1
  \qquad (a,c\in\Fp^r).
\end{equation}
The residue classes $\{X^a:a\in\Fp^r\}$ form an $\Fp$-basis of
$R$. Thus every $F\in R$ has a unique expansion
$F=\sum_{a\in\Fp^r}c_aX^a$. The $X$-monomials record sums in the group, while the $T$-monomials
record degree in a truncated polynomial ring. The following lemma
relates these two descriptions.

\begin{lemma}\label{lem:moments}
For the ring and ideal in~\eqref{eq:ring}, we have $J^{rp}=0$. Moreover, suppose that $F=\sum_{a\in\Fp^r}c_aX^a\in R$ satisfies
  $\sum_{a\in\Fp^r}c_a f(a)=0$
  for every $f\in\Fp[z_1,\ldots,z_r]$
   with $\deg f\leq r-1$.
Then $F\in J^r$.
\end{lemma}

\begin{proof}
Put $T_i=X_i-1$. Since the characteristic is $p$,
$X_i^p-1=(X_i-1)^p=T_i^p$, and hence
\begin{equation}\label{eq:truncated-ring}
  R\cong\Fp[T_1,\ldots,T_r]/(T_1^p,\ldots,T_r^p),
  \qquad J=(T_1,\ldots,T_r).
\end{equation}
The monomials
$T^\alpha=T_1^{\alpha_1}\cdots T_r^{\alpha_r}$, with
$0\leq\alpha_i<p$, form a basis of this quotient. Their total
degrees are at most $r(p-1)$. The ideal $J^m$ is spanned by those
basis monomials whose total degrees are at least $m$. It follows
that $J^{rp}=0$. In particular, every product of $p$ elements of $J^r$
is zero.

It remains to prove the last assertion. For $0\leq k<p$, define
the polynomial
\[
  \binom{z}{k}=\frac{z(z-1)\cdots(z-k+1)}{k!}\in\Fp[z],
  \qquad \binom{z}{0}=1.
\]
This is well defined because $k!$ is invertible in $\Fp$. Expanding
$X^a$ in the $T$-basis gives
\begin{equation}\label{eq:binomial-expansion}
  X^a=\prod_{i=1}^{r}(1+T_i)^{a_i}
     =\sum_{0\leq\alpha_1,\ldots,\alpha_r<p}
       \left(\prod_{i=1}^{r}\binom{a_i}{\alpha_i}\right)
       T^\alpha,
\end{equation}
where integer binomial coefficients are read in $\Fp$.
Consequently, the coefficient of $T^\alpha$ in $F$ is
\begin{equation}\label{eq:T-coefficient}
  [T^\alpha]F
  =\sum_{a\in\Fp^r}c_a
       \prod_{i=1}^{r}\binom{a_i}{\alpha_i}.
\end{equation}
The product on the right is the value at $a$ of the polynomial
$\prod_{i=1}^{r}\binom{z_i}{\alpha_i}$, whose total degree is
$|\alpha|=\alpha_1+\cdots+\alpha_r$. By
the condition on $F$, the expression in
\eqref{eq:T-coefficient} is zero whenever $|\alpha|\leq r-1$.
Thus only $T$-monomials of total degree at least $r$ occur in $F$,
which means that $F\in J^r$.
\end{proof}


For elements $y_1,\ldots,y_m$ of a commutative ring, write
\[
  e_k(y_1,\ldots,y_m)
  =\sum_{\substack{A\subseteq[1,m]\\|A|=k}}\prod_{j\in A}y_j,
  \qquad e_0=1.
\]
The following identity will allow us to use the congruences
$F_j\equiv F_1\pmod{J^r}$.

\begin{lemma}\label{lem:symmetric}
Let $A$ be a commutative $\Fp$-algebra, and let $I$ be an ideal of
$A$ with $I^p=0$. If $y_j=u+\eta_j$, where $u\in A$ and
$\eta_j\in I$ for $j\in[1,2p-1]$, then $e_p(y_1,\ldots,y_{2p-1})=u^p$.
\end{lemma}

\begin{proof}
Expand the left side according to the number of factors chosen
from the terms $\eta_j$. This gives
\begin{equation}\label{eq:symmetric-expansion}
\begin{split}
  e_p(u+\eta_1,\ldots,u+\eta_{2p-1})
  =\sum_{k=0}^{p}\binom{2p-1-k}{p-k}
    u^{p-k}e_k(\eta_1,\ldots,\eta_{2p-1}).
\end{split}
\end{equation}
Indeed, after fixing $k$ indices from which the error terms are
chosen, the remaining $p-k$ indices can be chosen in
$\binom{2p-1-k}{p-k}$ ways.

For $k=0$, we have $\binom{2p-1}{p}=1$.
For $1\leq k\leq p-1$, $\binom{2p-1-k}{p-k}=0$. Finally,
$e_p(\eta_1,\ldots,\eta_{2p-1})=0$, because each of its summands
is a product of $p$ elements of $I$. Thus only $u^p$ remains in
\eqref{eq:symmetric-expansion}.
\end{proof}

\section{Proof of the main theorem}\label{sec:proof}

We first prove the following crucial bound for the prime case.

\begin{proposition}\label{prop:prime}
Let $p$ be a prime and let $r\in\N$. Then $\EGZ(\mathbb Z_p^r)\leq\binom{2r}{r}(p-1)+1$.
\end{proposition}

\begin{proof}
Let $\cP_{r-1}=\{f\in\Fp[z_1,\ldots,z_r]:\deg f\leq r-1\}$,
and denote $b=\dim_{\Fp}\cP_{r-1}=\binom{2r-1}{r}.$ Let $S=a_1\cdot\ldots\cdot a_N$ be a sequence over $\mathbb Z_p^r$ with
$N\geq2b(p-1)+1$. We shall show that $S$ contains a zero-sum
subsequence of length $p$.

We first construct disjoint blocks with matching moments. Choose
a basis $1,f_1,\ldots,f_{b-1}$ of $\cP_{r-1}$ and, for each
$i\in[1,N]$, put
\[
  v_i=(f_1(a_i),\ldots,f_{b-1}(a_i))\in\Fp^{b-1}.
\]
By Lemma~\ref{lem:affine}, the indices can be
partitioned into $M$ nonempty blocks $B_1,\ldots,B_M$, each of size at most $b$.
Since $M\geq\lceil N/b\rceil\geq 2p-1$, we consider the first $2p-1$ blocks. Choose $i_*\in B_{2p-1}$,
and write $a_*=a_{i_*}$ and $v^*=v_{i_*}$. For each $j\in[1,2p-1]$, choose coefficients
$\lambda_{j,i}\in\Fp$ such that
\begin{equation}\label{eq:affine-weights}
 \sum_{i\in B_j}\lambda_{j,i}=1,\qquad
 \sum_{i\in B_j}\lambda_{j,i}v_i=v^*.
\end{equation}
For the last block, take $\lambda_{2p-1,i_*}=1$ and
$\lambda_{2p-1,i}=0$ for $i\in B_{2p-1}\setminus\{i_*\}$.
Since $1,f_1,\ldots,f_{b-1}$ is a basis, we obtain
\begin{equation}\label{eq:matching-moments}
 \sum_{i\in B_j}\lambda_{j,i}f(a_i)=f(a_*)
 \qquad\text{for every }f\in\cP_{r-1}
 \text{ and }j\in[1,2p-1].
\end{equation}

 In the ring $R$ of
\eqref{eq:ring}, define $F_j=\sum_{i\in B_j}\lambda_{j,i}X^{a_i}  (j\in[1,2p-1])$.
In particular, $F_{2p-1}=X^{a_*}$.
By~\eqref{eq:matching-moments} and Lemma~\ref{lem:moments},
\begin{equation}\label{eq:anchor-congruence}
 F_j-X^{a_*}\in J^r\qquad (j\in[1,2p-1]).
\end{equation}
If a vector occurs more than once, its coefficients are collected
before applying that lemma. Since $J^{rp}=0$,
Lemma~\ref{lem:symmetric}, with $u=X^{a_*}$ and $I=J^r$,
yields
\begin{equation}\label{eq:ep-identity}
 \sum_{\substack{A\subseteq[1,2p-1]\\|A|=p}}
       \prod_{j\in A}F_j
 =(X^{a_*})^p=1.
\end{equation}
    It remains to interpret the coefficient of $1$. For a fixed
$A\subseteq[1,2p-1]$ with $|A|=p$, expansion gives
\begin{equation}\label{eq:selection-expansion}
 \prod_{j\in A}F_j
 =\sum_{\substack{i_j\in B_j\\j\in A}}
   \left(\prod_{j\in A}\lambda_{j,i_j}\right)
   X^{\sum_{j\in A}a_{i_j}}.
\end{equation}
The monomial in a summand is $1$ if and only if
$\sum_{j\in A}a_{i_j}=0$ in $\Fp^r$.
The blocks are disjoint, so each selection uses $p$ distinct
occurrences of $S$. If $S$ had no zero-sum subsequence of length
$p$, every product on the left of~\eqref{eq:ep-identity}
would have coefficient of $1$ equal to zero.
This contradicts the right side, whose coefficient of $1$ is $1$.
Finally, $2b=\binom{2r}{r}$ gives the asserted bound.
\end{proof}

We use the following standard inductive inequality; see
\cite[Lemma~2.1(2)]{FGZ} and
\cite[(2)]{AD}.


\begin{lemma}\label{lem:extraction}
Let $a,b,r\in\N$. Then $\EGZ(\mathbb{Z}_{ab}^r)\leq a\bigl(\EGZ(\mathbb{Z}_b^r)-1\bigr)+\EGZ(\mathbb{Z}_a^r)$.

\end{lemma}

\begin{proof}[Proof of Theorem~\ref{thm:main}]
Fix $r\in\N$ and put $K_r=\binom{2r}{r}$. We prove $\EGZ(\mathbb{Z}_n^r)\leq K_r(n-1)+1$
by strong induction on $n$. The case $n=1$ is immediate, and the
case of prime $n$ follows from Proposition~\ref{prop:prime}.
Suppose that $n=pm$, where $p$ is a prime and $m\geq2$.
By Lemma~\ref{lem:extraction}, the induction hypothesis, and
Proposition~\ref{prop:prime}, we obtain
\begin{align*}
  \EGZ(\mathbb{Z}_{pm}^r)\leq p\bigl(\EGZ(\mathbb{Z}_m^r)-1\bigr)+\EGZ(\mathbb{Z}_p^r)\leq K_r(pm-1)+1.
\end{align*}
This completes the induction. Finally, $K_r(n-1)+1\leq4^r n$. The theorem follows.
\end{proof}

\medskip
\noindent\textbf{Statement of AI.} During the preparation of this work, the authors used ChatGPT to suggest simplifications for certain proof strategies and polish some statements. The authors are completely responsible for the content of the paper.

\medskip
\noindent\textbf{Acknowledgment.} This work was supported by the National Natural Science Foundation of China Grant No.~12571578 and Guangdong
Basic and Applied Basic Research Foundation Grant No.~2024A1515012564.


\end{document}